\documentclass[11pt,reqno]{amsart}
\usepackage{amssymb,amsmath,amscd,mathrsfs,amscd}
\usepackage{color}

\usepackage{graphicx}
\usepackage{epstopdf}
\usepackage{graphics}

\theoremstyle{plain}
\newtheorem{theorem}{Theorem}[section]

\newtheorem{lemma}[theorem]{Lemma}
\newtheorem{proposition}[theorem]{Proposition}
\newtheorem{corollary}[theorem]{Corollary}

\theoremstyle{definition}
\newtheorem{definition}[theorem]{Definition}
\newtheorem{remark}[theorem]{Remark}

\newtheorem{question}[theorem]{Question}

\newtheorem{example}[theorem]{Example}

\numberwithin{equation}{section}

\newcommand{\todo}[1]{\vspace{5 mm}\par \noindent
\marginpar{\textsc{}}
\framebox{\begin{minipage}[c]{0.95 \textwidth}
\tt #1 \end{minipage}}\vspace{5 mm}\par}

\newcommand\des{\mathop{\rm des}}
\newcommand{\Des}{\operatorname{Des}}

\newcommand{\la}{{\lambda}}
\newcommand{\g}{{\gamma}}

\newcommand{\be}{{\beta}}

\newcommand\nneg{\mathop{ \rm neg}}
\newcommand{\Colrm}{{\rm Col}}

\newcommand{\CS}{{\NN_0^{(r,n)}}}

\def\s{ \sigma }

\def\ZZ{{\mathbb Z}}
\def\NN{{\mathbb N}}

\def\k+1th{(k+1)^{th}}

 \DeclareMathOperator{\fmaj}{fmaj}
\DeclareMathOperator{\col}{col}

\def\fmaj{\mathop{\rm fmaj}\nolimits}

\def\imaj{\mathop{\rm imaj}\nolimits}
\def\icol{\mathop{\rm icol}\nolimits}
 \def\maj{\mathop{\rm maj}\nolimits}

\def\inv{\mathop{\rm inv}\nolimits}
\def\ides{\mathop{\rm ides}\nolimits}
\def\des{\mathop{\rm des}\nolimits}

\def\il{\bigl]\kern-.55em\bigl]}
\def\ir{\bigr]\kern-.55em\bigr]}

\title{Enumerating Wreath Products Via Garsia-Gessel Bijections }
\author{Riccardo Biagioli and Jiang Zeng}
\address{Universit\'e de Lyon, Universit\'e Lyon 1, Institut Camille
Jordan, UMR 5208 du CNRS, F-69622, Villeurbanne Cedex, France}
\email{biagioli@math.univ-lyon1.fr, zeng@math.univ-lyon1.fr}
\thanks{The two authors are supported by
the grant  ANR-08-BLAN-0243-03 and by the Program P2R franco-isra\'elien Math\'ematiques}
\begin{document}
\begin{abstract}
We generalize two bijections due to Garsia and Gessel to  compute the generating functions of the two vector statistics $(\des_G, \maj,\ell_G, \col)$ and $(\des_G, \ides_G, \maj, \imaj, \col, \icol)$ over the wreath product of a symmetric group by a cyclic group. Here  $\des_G$, $\ell_G$, $\maj$, $\col$, $\ides_G$, $\imaj_G$, and $\icol$ denote 
the number  of descents, length,  major index, color weight, inverse descents, inverse major index, and 
 inverse color weight, respectively.  Our main formulas generalize and unify  several known identities 
due to Brenti, Carlitz, Chow-Gessel, Garsia-Gessel, and Reiner on various distributions of statistics over Coxeter groups of type $A$ and $B$. 
\end{abstract}
\maketitle
\tableofcontents

\section{Introduction}
 Permutation statistics have a very natural setting within the theory of partitions as observed by 
 MacMahon~\cite{mac60}, Gordon~\cite{go63}, Stanley~\cite{s97}, Gessel~\cite{ge-thesis}, and Reiner~\cite{re93}, among others. This was made even more clear thanks to the work of Garsia and Gessel~\cite{gg79} who gave two remarkable bijections. The first one between sequences and pairs made of a permutation and a partition~\cite[\S 1]{gg79};   the second one between bipartite partitions and triplets made of a permutation and a coupe of special partitions~\cite[\S 2]{gg79}.  Thanks to these two bijections they were able to give elegant computations of the generating series 
 \begin{align}\label{e:gg1}
\sum_{n\geq 0}\frac{\sum_{\s\in S_n}t^{\des(\s)}q^{\maj(\s)}p^{\inv(\s)}}
{\prod_{i=0}^n(1-tq^i)}  \frac{u^n}{[n]_{p}!}=
\sum_{k\geq 0}t^k{e}[u]_{p}e[qu]_{p}\cdots e[q^ku]_{p},
 \end{align}
 and
 \begin{align}\label{e:gg2}
 \sum_{n\geq 0}\frac{u^n}{(t_1;q_1)_{n+1}(t_2;q_2)_{n+1}}
&\sum_{\s \in S_{n}}
t_1^{\des(\s)}t_2^{\des(\s^{-1})}q_1^{\maj(\s)}
q_2^{\maj(\s^{-1})} \nonumber\\
&=\sum_{k_1,k_2\geq 0}\frac{t_1^{k_1}t_2^{k_2}}
{(u ; q_1,q_2)_{k_1+1,k_2+1}}.
\end{align}
Here $\des$, $\maj$, and $\inv$ denote the number of descents, the
major index, and the number of inversions over the symmetric group $S_n$.

Extensions of the above bijections were given in \cite{berbia06,bl06,bc04,re93}.
In particular, Reiner \cite{re93, re93bis} generalized Garsia and Gessel's work 
to the hyperoctahedral group $B_n$, using $P$-partition theory. He obtained analogues of \eqref{e:gg1} and \eqref{e:gg2} for $B_n$ and $G(r,n)$, the wreath product of the symmetric group $S_{n}$ with the cyclic group $\ZZ_{r}$ (cf. \cite[Corollary 7.2 and 7.3]{re93}). 

In this paper, we give $G(r,n)$-analogues of the two Garsia and Gessel bijections (cf. Proposition~\ref{l:bijection} and Theorem~\ref{ggg}). 
A fundamental notion will be that of $\g$-compatible partitions, which generalizes that of $\s$-compatible partitions introduced by Garsia and Gessel in~\cite{gg79}. We use our bijections to give two different extensions of \eqref{e:gg1} and \eqref{e:gg2} for $G(r,n)$, by computing 
the generating functions of the two vector statistics $(\des_G, \maj,\ell_G, \col)$ and $(\des_G, \ides_G, \maj, \imaj, \col, \icol)$, whose definitions will be given in the next sections (cf. Theorem~\ref{teq:wreath} and Theorem~\ref{5stat}).

The aforementioned Reiner's results also compute similar generating functions but his definitions of the length and the major index are slightly different from ours due to a different choice of the generating set for $G(r,n)$. In the case of the hyperoctahedral group, more details explaining differences and relations between our and Reiner's results are given in a separate paper~\cite{bz09}, where connections with the work \cite{cg07} of Chow-Gessel are also established.

The choice of our statistics is motivated by the aim to take into account  the flag-major index introduced by Adin and Roichman in~\cite{ar01}. Indeed Theorem~\ref{teq:wreath} unify and generalize several known results due to Brenti, Carlitz, Chow-Gessel, Gessel, and Reiner, on various distributions of statistics over Coxeter groups of type $A$ and $B$.
Moreover, Theorem~\ref{5stat} will allow us to give an explicit description of the generating function of the Hilbert series of some $G(r,n)$-invariant algebras, studied by Adin and Roichman~\cite{ar01}, and involving the flag-major index.

Clearly, if we set $r=2$ in the identities given in Theorem~\ref{teq:wreath} and Theorem~\ref{5stat}, we find the analogous results for the hyperoctahedral group $B_{n}$. We point out that these identities are actually equivalent, and we show how to recover the general $G(r,n)$-case from the knowledge of the $B_{n}$-case. 

\section{Definitions and notation}
In this section we give some definitions, notation and results that will be used in the rest of this work. For $n \in \NN$ we let $[n]:= \{ 1,2, \ldots , n \} $ (where $[0]:= \emptyset $). Given $n, m \in \ZZ, \; n \leq m$, we let $[n,m]:=\{n,n+1, \ldots, m \}.$ The cardinality of a set $A$ will be denoted either by $|A|$ or by $\#A$.
Let $P$ be a statement: the characteristic function $\chi$ of $P$ is defined as $\chi(P)=1$ if $P$ is true, and $0$ otherwise.
As usual for $n \in \NN$, we let
\begin{eqnarray*}
(a;p)_n:=\left\{\begin{array}{ll} \;\; 1, & {\rm  if} \ n=0;\\
(1-a)(1-ap)\cdots (1-ap^{n-1}),& {\rm if} \ n\geq 1.
 \end{array}\right.
\end{eqnarray*}
Moreover, for $n, m \in \NN$ we let
$$(a;p,q)_{n,m}:=\left\{\begin{array}{ll} \;\; 1,
 & {\rm  if} \ n \ {\rm or} \ m \ {\rm are \ zero};\\
{\displaystyle\prod_{1\leq i\leq n}\prod_{1\leq j\leq m} (1-ap^{i-1}q^{j-1})}, & {\rm if} \ n,m \geq 1.
\end{array}\right.$$
For our study we need notation for $p$-analogues of integers and factorials. 
These are defined by the following expressions
\begin{align*}
[n]_p:&=1+p+p^2+\ldots + p^{n-1}, \\
[n]_p!:&=[n]_p[n-1]_p\cdots[2]_p[1]_p, \\
\ [\hat{n}]_{a,p}!:&=(-ap;p)_n [n]_p!,
\end{align*}
where $[0]_p!:=1$. For $n=n_0+n_1+\cdots +n_k$ with $n_{0}, \ldots, n_{k}\geq 0$ we define the $p$-analogue of
 multinomial coefficient by
$$
{\hat n\brack \hat n_0,n_1,\ldots, n_k}_{a,p}:=\frac{[\hat n]_{a,p}!}{[\hat n_0]_{a,p}![n_1]_p!\cdots [n_k]_p!}.
$$
Finally, 
\begin{equation}\label{exp}
e[u]_p:=\sum_{n\geq 0}\frac{u^n}{[n]_p !}, \quad {\rm and}  \quad
\hat e[u]_{a,p}:=\sum_{n\geq 0}\frac{u^n}{[\hat n]_{a,p}!}
\end{equation}
are a classical $p$-analogue and a $(a,p)$-analogue of the exponential function. 

\bigskip

Let $S_n$ be the symmetric group on
$[n]$. A permutation $\s \in S_n$ will be denoted by $\s=\s(1)\cdots
\s(n).$

Let $r,n \in \mathbb{P}$. The {\em wreath product} $\ZZ_r \wr S_n$ of
$\mathbb{Z}_r$ by $S_n$ is defined by
\begin{equation}\label{def-grn}
G(r,n):=\{(c_1,\ldots,c_n;\sigma) \mid c_i \in [0,r-1],\sigma \in
S_n\}.
\end{equation}
Any $c_i$ can be considered as the color
of the corresponding entry $\s(i)$. This is why
this group is also called the {\em group of r-colored
permutations}. Sometimes we will represent its elements in {\em window
notation} as
\[\g=[\g(1),\ldots, \g(n)]=[\s(1)^{c_1},\ldots, \s(n)^{c_n}].\]
Sometimes we will call  $\sigma(i)$ the {\em absolute value} of $\g(i)$, denoted $|\g(i)|$.
When it is not clear from the context, we will denote $c_i$ by
$c_i(\g)$. Moreover, if $c_i=0$, it will be omitted in the window notation of
$\g$. We denote by
\[\Colrm(\g):=(c_1,\ldots,c_n) \;\;\; {\rm and} \; \;\; \col(\g): =\sum_{i=1}^n c_i,\]
the {\em color vector} and the {\em color weight} of any $\g:=(c_1,\ldots,c_n; \s) \in G(r,n)$. For example, if $\g=[4^1, 3, 2^4, 1^2] \in G(5,4)$ then $\Colrm(\g)=(1,0,4,2)$ and $\col(\g)=7$.
\smallskip

The product in $G(r,n)$ is defined as follows: 
$$(c_1,\ldots,c_n;\s)\cdot (d_1,\ldots,d_n;\tau):=(d_1+c_{\s(1)},\ldots,d_n+c_{\s(n)};\s\tau),
$$
where the product $\sigma \tau$ is from right to left as usual.

  Clearly for a colored permutation $\g=({c_1},\ldots, {c_n}; \sigma)\in G(r,n)$ the {\em inverse colored permutation} is given by $\g^{-1}=(c'_1,\ldots, c'_n;\sigma^{-1})$, where
\begin{equation}\label{inverse}
\left\{\begin{array}{ll} c'_i=c_{\sigma^{-1}(i)}, & {\rm if} \ c_i=0; \\
                                     c'_i=r-c_{\sigma^{-1}(i)}, & {\rm otherwise}. \end{array}
                                     \right.
                                     \end{equation}
In this paper we will use mostly a different type of inverse. We define the {\em skew inverse 
permutation} of $\g$ by
\begin{align}\label{skew}
\tilde{\g}^{-1}:=(c_{\sigma^{-1}(1)},\ldots, c_{\sigma^{-1}(n)};\sigma^{-1}).
\end{align}
For example, if $\g=[3,6^1,4^3,7^2,2^1, 1,5] \in G(4,7)$ then 
$\tilde{\g}^{-1}=[6, 5^1, 1, 3^3, 7, 2^1,4^2]$, while $\g^{-1}=[6, 5^3, 1, 3^1, 7, 2^3,4^2]$. Note that  when $G(r,n)$ is a Coxeter group, the skew inverse is actually the inverse.
\smallskip

The group $G(r,n)$ is generated by the set $S_G:=\{s_0,s_1,\ldots,s_{n-1}\}$ where for
$i \in [n-1]$ 
\begin{equation}\label{gen-set}
s_i:=[1,\ldots,i-1,i+1,i,i+2,\ldots,n] \;\; {\rm and} \;\; s_0:=[1^1,2,\ldots,n],
\end{equation}
with relations given by the following Dynkin-like diagram
\begin{figure}[htdb!]
\centering
\includegraphics[scale=.6]{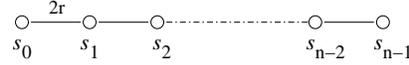}
\caption{The  Dynkin-like diagram of $G(r,n)$}\label{graphG}
\end{figure}

\begin{definition} In all the paper we will use the following order  
\begin{equation}\label{order}
n^{r-1}< \ldots < n^1< \ldots < 1^{r-1}< \ldots < 1^1 < 0 < 1 < \ldots < n 
\end{equation}
on the set $\{ 0, 1,\ldots, n, 1^1,\ldots, n^1,\ldots, 1^{r-1},\ldots, n^{r-1}\}$ of colored integers.
\end{definition}

The following characterization of the {\em length} of $\g=\s(1)^{c_1}\cdots\s(n)^{c_n} \in G(r,n)$ is well-known (see e.g., \cite{re93, stein94}, \cite[Theorem 4.3]{ba04})
\begin{equation}\label{d:length}
\ell_G(\g)=\inv(\g)+\sum_{c_i\neq 0} \left(\s(i) +c_i-1\right),
\end{equation}
where the {\em inversion number} is defined by
\begin{eqnarray*}
\inv(\g) & := & |\{ (i,j) \in [n]\times [n] \mid \g(i)>\g(j) \}|.
\end{eqnarray*}
The generating function for the length is given by
\begin{equation}\label{gf:lengthG}
\sum_{\g \in G(r,n)} p^{\ell_G(\g)}=[n]_p! (-p[r-1]_p;p)_n=[\hat{n}]_{[r-1]_p,p}!.
\end{equation}
\smallskip

The {\em descent set} of $\g \in G(r,n)$ is defined by 
\begin{eqnarray}\label{des}
\Des_G(\g)
&:=&\{i \in [0,n-1] \mid \g(i) > \g(i+1)\},
\end{eqnarray}
where  $\g(0):=0$, and its cardinality is denoted by $\des_G(\g)$.  Note that 
$0\in \Des_G(\g)$ if and only if $c_{1}(\g)>0$.
\smallskip

As usual the {\em major index}  is defined  to be the sum of descent positions:
$$ \maj(\g)=\sum_{i\in \Des_G(\g)}i,$$
and the \emph{flag-major index} (see \cite{ar01}) is defined by $$\fmaj(\g):=r \cdot \maj(\g)+ \col(\g).$$
\smallskip
For example, for  $\g=[4^1, 3, 2^4, 1^2] \in G(5,4)$ we have $\inv(\g)=2$, 
$\ell_{G}(\g)=13$,  $\Des_{G}(\g)=\{0,1\}$, $\des_{G}(\g)=2$, $\maj(\g)=2$, and $\fmaj(\g)=17$.

\section{Encoding colored sequences}
\label{bijection1}

In this section  we collect some notions
and results that  will be needed in the rest of this paper.
First of all
we generalize Garsia-Gessel and Reiner's method of
encoding sequences and signed sequences (see \cite[\S 1]{gg79} and  \cite[\S 2]{re93}) to colored sequences. 

\medskip

Let ${\mathcal P}_n$ be the set of nondecreasing sequences of nonnegative integers $(\lambda_1, \lambda_2, \ldots, \lambda_n)$, that is partitions of length less than or equal to $n$. 
\smallskip

Define by $\mathbb{N}^{(r,n)}$ the set of $n$-tuples $f=(f_1^{c_1},\ldots,f_n^{c_n})$ of colored integers, where $f_i \in \NN$, and $c_i \in [0,r-1]$. We will mostly consider its subset 
$$\CS:=\left\{(f_1^{c_1},\ldots,f_n^{c_n})\mid f_i \in \NN, c_i \in [0,r-1] : \ {\rm if} \ f_i=0 \ {\rm then} \ c_i=0\right\}.$$

When it is not clear from the context, we will denote $c_i$ by $c_i(f)$. For $f= (f_1^{c_1},\ldots,f_n^{c_n}) \in \mathbb{N}^{(r,n)}$, we define
\begin{align}\label{d:stat1}
\max(f):=\max_{i \in [n]}\{f_i\} \quad {\rm and} \quad |f|:=\sum_{i=1}^n f_i.
\end{align}

\begin{definition}[The colored permutation $\pi(f)$]\label{pif} 
Given a colored sequence $f=(f_1^{c_1},\ldots,f_n^{c_n}) \in \CS$, we construct the colored permutation $\pi(f)$ as follows. For $\nu \in \NN$, we first define the sets 
\begin{equation}\label{Anu}
A_\nu:=\{i^{c_i} \mid f_i =\nu \ {\rm and} \ c_i=c_i(f)\},
\end{equation}
and arrange the entries in each nonempty ``bloc'' $A_\nu$ in increasing order (cf. (\ref{order})), obtaining $\uparrow \! A_\nu$. Then 
$$
\pi(f):=[\uparrow \! A_{\nu_{1}}, \uparrow \!A_{\nu_{2}}, \ldots,  \uparrow\! A_{\nu_{k}}]
$$ 
is obtained by juxtaposing the entries of the nonempty blocs  $\uparrow\! A_{\nu_{1}}, \uparrow\! A_{\nu_{2}}, \ldots$, where
$\nu_{1}<\nu_{2}<\cdots$.
\end{definition}

\begin{example}\label{ex-pif} Let  $f=(4^2,4^1,1,3^3,6,3^1,4^2)\in \NN^{(4,7)}_0$. Then 
$A_1=\{3\}, A_3=\{4^3,6^1\}, A_4=\{1^2,2^1,7^2\}$, and $A_6=\{5\}.$ Hence $\uparrow\! A_1=\{3\}, \uparrow\!A_3=\{6^1,4^3\}, \uparrow\! A_4=\{7^2,2^1,1^2\}$, $\uparrow\! A_6=\{5\}$, and
$$\pi(f)=[\uparrow \! A_1, \uparrow \!A_3, \uparrow\! A_4, \uparrow\! A_6]=[3,6^1,4^3,7^2,2^1,1^2,5] \in G(4,7).$$
\end{example}

The following equivalent characterization of the colored permutation $\pi(f)$ can be derived from Definition~\ref{pif}, and will be
frequently  used along the paper. 

\begin{proposition}\label{def:bijection}
Let  $f=(f_1^{c_1},\ldots,f_n^{c_n}) \in \CS$. Then $\pi(f)=\gamma=(c_1,\ldots,c_n;\sigma)$ is the unique colored permutation satisfying: 
\begin{enumerate}
\item[1)] $f_{\s(1)}\leq f_{\s(2)}\leq \ldots \leq f_{\s(n)}$;
\item[2)] $c_i(\g)=c_{\s(i)}(f)$;
\item[3)] If  $f_{\s(i)}=f_{\s(i+1)}$, then $\g(i)<\g(i+1)$.
\end{enumerate}
\end{proposition}
 
\begin{definition}
Let $\g=(c_1,\ldots, c_n;\s) \in G(r,n)$. We say that a sequence $f \in \CS$ is {\em associated with $\g$} if  $\pi(f)=\g$. 
\end{definition}
We remark that if $f \in \CS$ is associated with $\g$, then by parts 1) and 3) of Proposition~\ref{def:bijection} we have
\begin{equation}\label{pro-ass}
i \in \Des_G(\g) \Longrightarrow f_{\s(i)}<f_{\s(i+1)}.
\end{equation}

\begin{lemma}[The partition $\lambda(f)$]\label{laf} 
Given a colored sequence $f=(f_1^{c_1},\ldots,f_n^{c_n}) \in \CS$, define  
\begin{equation}\label{etiquette}
\lambda_i:=f_{\s(i)}-|\{j\in {\Des}_G(\g) \mid j\leq i-1\}|\quad\text{for} \quad
1\leq i\leq n,
\end{equation}
where $\pi(f)=\g=(c_1,\ldots,c_n;\sigma)$. Then the sequence $\la(f):=(\la_1,\ldots,\la_n)$ is a partition.
\end{lemma}
\begin{proof} If $0 \in {\Des}_G(\g)$ then $c_1(\g)>0$. Since $c_1(\g)=c_{\s(1)}(f)$ and $f \in \CS$, this implies that $f_{\s(1)} > 0$. Hence $\lambda_1\geq 0$. If $i > 1$, 
$$\lambda_{i+1}-\lambda_i = f_{\s(i+1)}-f_{\s(i)}-\chi(i\in \Des_G(\g)).$$
Since $f$ is associated with $\g$, if $i \in {\Des}_G(\g)$ then $f_{\s(i)}<f_{\s(i+1)}$. Hence $\lambda_{i+1}-\lambda_i \geq 0$ and $\lambda \in \mathcal{P}_n$.
\end{proof}

\begin{example}
Let $f=(4^2,4^1,1,3^3,6,3^1,4^2)\in \NN^{(4,7)}_0$ be the sequence of Example~\ref{ex-pif}. Then $(f_{\s(1)}, \ldots,f_{\s(n)})=(1,3,3,4,4,4,6)$. Since $\Des_G(\pi(f))=\{1,3\}$, we have $\lambda=(1,2,2,2,2,2,4)$.
\end{example}

\begin{definition}[The colored sequence $\lambda^\g$]\label{lambdagamma}
Given a partition $\lambda \in \mathcal{P}_n$ and a colored permutation $\g=({c_1},\ldots, {c_n}; \sigma)\in G(r,n)$ we denote by $\lambda^\g$ the following colored sequence
\begin{align}\label{lagamma}
\lambda^\g:=(\lambda_{\sigma(1)}^{c_1}, \ldots, \lambda_{\sigma(n)}^{c_n})\in \mathbb{N}^{(r,n)}.
\end{align}
\end{definition}

\begin{proposition}\label{l:bijection}
The map $f \mapsto (\pi(f), \lambda(f))$ is a bijection between $\CS$ and $G(r,n)\times \mathcal{P}_n$. Moreover,  if $\g:=\pi(f)$ and $\lambda:=\lambda(f)$ then
\begin{align}
 \max(f)&=\max(\lambda)+\des_G(\g), \label{eq:1} \\
 |f|&=|\lambda|+n\des_G(\g) -\maj(\g).\label{eq:2}
\end{align}
\end{proposition}
\begin{proof} To see that this map is a bijection, we construct its inverse. Starting from $\g=(c_1,\ldots,c_n;\s)$ and $\lambda$ in $G(r,n) \times {\mathcal P}_n$, we denote by
$\mu=(\mu_1,\ldots,\mu_n)$ the partition 
\begin{equation}\label{def:mu}
\mu_i:= \lambda_i + |\{j \in {\Des}_G(\g) \mid j\leq i-1\}|.
\end{equation}
Then 
we define  $f=(f_1^{c_1},\ldots,f_n^{c_n}) \in \mathbb{N}^{(r,n)}$  by letting 
\begin{eqnarray}\label{f-1}
f:=\mu^{\tilde{\g}^{-1}} \quad {\rm that \ is,  \ for \ each} \ i \in [n] \quad 
\left\{ \begin{array}{ll} 
f_i =\mu_{\s^{-1}(i)} \\
c_i(f)=c_{\s^{-1}(i)}(\g). \end{array}\right.
\end{eqnarray}
Let us prove that the image of $f$ by this map is $(\g,\la)$.
 We have $\pi(f)=\g$ because $\g$ satisfies the three conditions of Proposition~\ref{def:bijection}:
\begin{itemize}
\item[1)]  since $f_{\sigma(i)}=\mu_{i}$ then $(f_{\s(1)},\ldots,f_{\s(n)})=\mu$ is nondecreasing; 
\item[2)]  $c_i(f)=c_{\s^{-1}(i)}(\g)\Longrightarrow c_{\s(i)}(f)=c_i(\g)$;
\item[3)]   if $f_{\s(i)}=f_{\s(i+1)}$ then $\mu_i=\mu_{i+1}$, and (\ref{def:mu}) implies that $i \not \in \Des_G(\g)$. 
\end{itemize}
It is clear that  $\la(f)=\la$. 
To see that $f \in \CS$, we  need to show  that $\mu_i=0$ implies $c_i(\g)=0$.
 If $\mu_i=0$, then by (\ref{def:mu}) $\la_i=0$ and $j \not\in \Des_G(\g)$ for all $j \leq i-1$. 
 In particular $0$ is not a descent of $\g$. Hence $c_1(\g)=\ldots=c_i(\g)=0$.
\smallskip

Equation~(\ref{eq:1}) follows from $\max(f)=f_{\s(n)}=\la_{\s(n)}+\des_G(\g)$. Since
$$ 
|f|-|\la|=\sum_{i=1}^n|\{j\in {\Des}_G(\g)|j\leq i-1\}|=\sum_{j\in {\Des}_G(\g)}(n-j)=n\des_G(\g)-\maj(\g),
$$
we have Equation~(\ref{eq:2}).
\end{proof}

\begin{example} This is an example of the construction in the above proof. Let $(\g,\lambda)$ be the pair
$$([5^1,3^1,1,2^2,4^2],(0,2,2,3,3))\in G(3,5)\times {\mathcal P}_5.$$
We find $\Des_G(\g)=\{0,1,3,4\}$, so $\mu=(1,4,4,6,7).$ 
Now $\tilde{\g}^{-1}=[3,4^2,2^1,5^2,1^1]$ and 
$$f:=\mu^{\tilde{\g}^{-1}}=(4,6^2,4^1,7^2,1) \in \NN^{(3,5)}_0.$$
\end{example}

\begin{definition}\label{gamma-compatible}
Analogously to the case of the symmetric group \cite[(2.12)]{gg79}, we say that a partition $\lambda$ is {\em $\g$-compatible} if $\lambda_i<\lambda_{i+1}$ for all $i \in \Des_G(\g)$, where we let $\la_0:=0$. 
\end{definition}
\noindent Notice that if $\lambda$ is $\g$-compatible and $0 \in \Des_G(\g)$, then $\la_1\geq 1$.
\smallskip

The following result gives a relation between $\g$-compatible partitions and sequences associated with $\g$, which will be used in the proof of Theorem~\ref{5stat}.

\begin{proposition}\label{compatible-associated}
Let $\mu \in \mathcal{P}_n$, and $\g=(c_1,\ldots,c_n;\s) \in G(r,n)$. Then 
$$\mu \ \mbox{ is} \ \mbox{ $\g$-compatible} \Longleftrightarrow \pi(\mu^{\tilde{\g}^{-1}})=\g.$$
\end{proposition}
\begin{proof}
To simplify the notation, as in (\ref{f-1}), we let $f:=\mu^{\tilde{\g}^{-1}}$ within this proof. 

Suppose that $\mu$ is $\g$-compatible. We show that $\g$ verifies all three conditions of Proposition~\ref{def:bijection}. The first one is clear since $\mu_i=f_{\s(i)}$, and $\mu$ is a partition. Secondly, by definition (see (\ref{skew})) $c_{\s^{-1}(i)}(\g)=c_i(\tilde{\g}^{-1})$. Hence  $c_i(\g)=c_{\s(i)}(\tilde{\g}^{-1})=c_{\s(i)}(f)$, where the last equality follows from Definition~\ref{lambdagamma}. Now suppose that $f_{\s(i)}=f_{\s(i+1)}$, that is $\mu_i=\mu_{i+1}$. The $\g$-compatibility of $\mu$ implies that $\g(i)<\g(i+1)$, so the third condition.

Conversely, if $\pi(f)=\g$ then condition 3) of Proposition~\ref{def:bijection} implies that $\mu$  is $\g$-compatible, since $\mu_i=f_{\s(i)}$, for every $i \in [n]$. In the case $0 \in \Des_G(\g)$, we let $\s(0)=0$, and $f_0=0$, as in Definition~\ref{gamma-compatible}.
\end{proof}

The next result will be a basic tool in the proofs of Theorem~\ref{eq:wreath} and Theorem~\ref{5stat}.
\begin{proposition}\label{desmaj}
Let $\g \in G(r,n)$. Then
\begin{align}\label{caseB}
\sum_{f \in \CS \mid \pi(f)=\g} t^{\max(f)}q^{\max(f)\cdot n-|f|}
=\frac{t^{\des_G(\g)}q^{\maj(\g)}}{(t;q)_n}.
\end{align}
\end{proposition}
\begin{proof} It follows from the bijection in Proposition~\ref{l:bijection}, (\ref{eq:1}), and (\ref{eq:2}) that 
\begin{align}
\sum_{f\in \CS \mid \pi(f)=\g}
t^{\max(f)}q^{|f|}&=\sum_{\lambda \in \mathcal{P}_n}t^{\max(\lambda)+\des_G(\g)}q^{|\lambda|+n\des_G(\g)-\maj(\g)}\nonumber \\
&=\frac{t^{\des_G(\g)}q^{n\des_G(\g)-\maj(\g)}}{(tq;q)_n}. \label{seq}
\end{align}
By replacing $q$ by $q^{-1}$ and $t$ by $tq^n$ in (\ref{seq}) we get the result.
\end{proof}

Setting $r=2$ we have $G(2,n)=B_n$ the hyperoctahedral group or group of signed permutations. 
As usual, the weight color will be  denoted by $\nneg$, the number of negative entries,
and  the descent set by $\Des_{B}$. \\
There is a natural projection of $G(r,n)$ into $B_n$ for every $r\geq 2$, which ``forgets the colors". 
\begin{definition}[Projection of $G(r,n)$ onto $B_n$]\label{projection}
We define
$\phi: G(r,n) \rightarrow B_n$ by sending $\g=({c_1},\ldots,{c_n}; \s)$ to 
$$
\phi(\g)=(\hat{c}_1,\ldots,{\hat{c}_n; \s)},\qquad\textrm{where}\qquad
\hat{c}_i:=\left\{\begin{array}{ll} 0 \ , & {\rm if} \quad c_i=0; \\
                                                      1 \ , & {\rm if} \quad c_i \geq 1.
                                                      \end{array}\right.
$$
It is easy to see that $\phi$ is not a group homomorphism, and that $\phi(\g^{-1})=\phi(\g)^{-1}$. The following lemma, relating the statistics of $G(r,n)$ with those of $B_n$, refines a result of Reiner~\cite[\S 7]{re93}. The proof is similar.
\end{definition}

\begin{lemma}\label{l:projection} We have 
\begin{itemize}
\item[1)] $\Des_G(\g)=\Des_B(\phi(\g))$;
\item[2)] $\Des_G(\g^{-1})=\Des_B(\phi(\g)^{-1})$;
\item[3)]  For any $\hat{\g} \in B_n$ we have 
$$\sum_{\g \in G(r,n) \mid \phi(\g)=\hat{\g}} p^{\ell_G(\g)}a^{\col(\g)}=p^{\ell_B(\hat{\g})}(a[r-1]_{ap})^{\nneg(\hat{\g})}.$$
\end{itemize}
\end{lemma}

It is clear that for the skew inverse we have
\begin{align}\label{phig}
\phi(\g^{-1})=\phi(\tilde{\g}^{-1}).
\end{align}

The next result follows from Proposition~\ref{desmaj} and part 1) of Lemma~\ref{l:projection}.
\begin{corollary}\label{l:desmaj}
Let $\g \in G(r,n)$. Then
\begin{align}
\sum_{f \in \CS \mid \pi(f)=\phi(\g)} t^{\max(f)}q^{\max(f)\cdot n-|f|}=
\sum_{f \in \CS \mid \pi(f)=\g} t^{\max(f)}q^{\max(f)\cdot n-|f|}.
\end{align}
\end{corollary}
%

\section{Quotients of $G(r,n)$}

There exists a well-known decomposition of any Coxeter group as the product of a quotient by a parabolic subgroup, see for example \cite[Proposition 2.4.4]{bb05}. The following analogue result for $G(r,n)$ is probably known.
Due to the lack of an adequate reference, we provide a proof.

\smallskip
Let  $S_G:=\{s_{0},\ldots, s_{{n-1}}\}$ be the generating set of $G(r,n)$ defined in (\ref{gen-set}). For 
$J\subseteq [0, n-1]$ we let $G_{J}:=\langle s_{i}\mid i\in J\rangle$ be the {\em parabolic subgroup}
 generated by $J$.
\begin{proposition}
\label{quotient}
 For $J \subseteq  [0, n-1]$ let $G_{J}$ be the  parabolic subgroup  of $G(r,n)$ generated by $J$, and
$$
G^{J} := \{ \tau \in G(r,n) \mid \Des_G(\g) \subseteq [0, n-1]\setminus J\},
$$
the {\em (right) quotient}. Then every $\g \in G(r,n)$ has a unique factorization
$\g=\tau \cdot \delta$, with $\tau \in G^{J}$ and $\delta \in G_{J}$, such that
\begin{equation}\label{length-color}
\ell_G(\g)=\ell_G(\tau)+\ell_G(\delta) \quad {\rm and} \quad  \col(\g)=\col(\tau)+\col(\delta).
\end{equation}
\end{proposition}

\begin{proof} 
We let $[0,n-1]\setminus J:=\{d_1,\ldots, d_k\}$, where $d_1<\ldots<d_k$.  We have two cases to consider.
\begin{itemize}
\item[a)]
Suppose  $0 \not\in J$, hence $d_1=0$.  This means that 
\begin{align*}
G_J&\simeq S[1, d_2] \times S[d_2+1, d_3] \times \ldots \times S[d_{k}+1,n],\quad\textrm{and}\\
 G^J&=\{\tau \in G(r,n) \mid \Des_G(\tau)\subseteq \{0,d_2,\ldots,d_k\}\},
\end{align*}
 where $S[a,b]$ denotes the symmetric group  of the set $[a,b]$.
We can write any element $\g \in G(r,n)$ as juxtaposition of $k$ words 

\begin{equation}\label{mots}
\g=\g^1 \mid  \g^2 \mid\ldots \mid \g^k,
\end{equation}
where $\g^i:=[\g(d_{i}+1),\ldots, \g(d_{i+1})]$, for $i=1,\ldots,k$, and $d_{k+1}:=n$.
We define 
\begin{equation}\label{taudelta}
\tau:=\uparrow\!\g^1\cdots \uparrow\! \g^k \quad {\rm and} \quad \delta=\delta^1\cdots \delta^k,
\end{equation}
where $\uparrow\! \g^i$ is the increasing arrangement of the letters in $\g^i$, (cf. Definition~\ref{pif}), and  $\delta^i$ is the unique permutation in $S[d_{i}+1,d_{i+1}]$ such that 
\begin{align}\label{action}
(\uparrow \! \g^i)\cdot {\delta^i}
=\g^i, \quad {\rm for} \quad i \in [k].
\end{align}

For example, suppose $\g=[5,2^2,4^1,3,1^1,6^2,8,7^2]\in G(3,8)$ and $J=\{1,2,4,5,7\}$. 
We have $d_{1}=0$, $d_{2}=3$, $d_{3}=6$,
and $\g^1\mid \g^2\mid \g^3=[5,2^2,4^1\mid 3, 1^1,6^2\mid 8,7^2]$. 
Then
$\tau=[4^1,2^2,5\mid 6^2,1^3,3 \mid 7^2,8]$,
 and $\delta=[3,2,1 \mid 6,5,4 \mid 8,7]$.
\smallskip

Clearly, from \eqref{action} and definition of product  follows that  $\g =\tau \cdot \delta$. This decomposition is unique due to the uniqueness
of $\tau$ and $\delta$. 

Since $\delta \in S_n$, and $\tau$ has  the same colored entries of $\g$
we have  
\begin{align}\label{coleurs}
\sum_{c_{i}(\g)\neq 0}(\sigma(i)+c_{i}-1)=\sum_{c_{i}(\tau)\neq 0}(|\tau(i)|+c_{i}-1),
\end{align}
where recall that $ |\tau(i)|$ stands for the absolute value of 
 $\tau(i)$. Therefore $ \col(\g)=\col(\tau)+\col(\delta)$.

For $i<j$, we denote $\inv(\g^i\mid \g^{j}):=|\{(a,b) \mid a \in \g^i, b \in \g^j\}|$ the number of  inversions between blocs.  Clearly we have 
\begin{equation}\label{sum-inversions}
\inv(\g)=\sum_{1\leq i\leq k}\inv(\g^{i})+\sum_{1\leq i<j\leq k} \inv(\g^{i}\mid \g^j).
\end{equation}
Note that
\begin{equation}\label{inv-blocs}
\inv(\uparrow\!\g^i)=\inv(\delta^i\mid \delta^{j})=0, \; \inv(\g^i)=\inv(\delta^i), \;  \textrm{and} \; \inv(\g^i\mid \g^{j})=\inv(\uparrow\!\g^i\mid\uparrow\!\g^{j}),
\end{equation}
for $1\leq i \leq k$ and $1\leq i<j\leq k$.
Thus by \eqref{sum-inversions} we have $\inv(\g)=\inv(\tau)+\inv(\delta)$. This, \eqref{coleurs}, and the definition (\ref{d:length}) imply $\ell_G(\g)=\ell_G(\tau)+\ell_G(\delta)$.
\smallskip

\item[b)] Now suppose $0\in J$, hence  $d_1>0$. Then
\begin{align*}
G_J&\simeq G(r,d_1)\times S[d_2+1, d_3] \times \ldots \times S[d_{k}+1,n],\quad\textrm{and}\\
 G^J&=\{\tau \in G(r,n) \mid \Des_G(\tau)\subseteq \{d_1,d_2,\ldots,d_k\}\},
  \end{align*}
 If $\g=[\s(1)^{c_1},\ldots,\s(n)^{c_n}] \in G(r,n)$, we let 
$$\tau:=\uparrow\! \s^1 \uparrow\! \g^2 \cdots \uparrow\! \g^k,$$
where $\uparrow\! \g^j$ is defined as in \eqref{mots}, and $\uparrow\! \s^1$ is the  increasing arrangement of the absolute values of the entries in $\g^1$. We define  $\delta^{1}\in G(r,d_{1})$ as the unique colored permutation such that $\g^1=\uparrow\! \s^1\cdot \delta^1$
and $\delta_{2}, \ldots, \delta_{k}$ as in (\ref{taudelta}).

For example, if $\g=[5,2^2,4^1,3,1^1,6^2]\in G(3,6)$, and $[0,5]\setminus J=\{0,3\}$, then $$\tau=[2,4,5 \mid 6^2,1^3,3] \quad {\rm and} \quad \delta=[3,1^2,2^1 \mid 6,5,4].$$
As before we have $\g=\tau\cdot \delta$, with $\tau \in G^J$ and $\delta \in G_J$, and the uniqueness of the decomposition. Since $0 \not\in \Des_G(\tau)$ and  $\uparrow\! \s^1$ is an increasing sequence, all entries in  $\uparrow\! \s^1$ are positive. Notice that $\delta^1$ is an order and color preserving reduction of $\g^1$. 

More precisely, if $\g^1=[a_1^{c_1},\ldots a_{d_1}^{c_{d_1}}]$ then
$\delta^1=[b_1^{c_1},\ldots b_{d_1}^{c_{d_1}}] \in G(r,d_1)$, 
and the two corresponding absolute value  sequences are ordered in the same way. 
In the above example we have $\g^1=[5,2^2,4^1]$ and $\delta^1=[3,1^2,2^1]$. \\
From this and the definitions of $\tau$ and $\delta$, it follows that 
$$\col(\g)=\col(\uparrow\! \g^2\mid\cdots \mid\uparrow\! \g^k)+\col(\delta^1)= \col(\tau)+\col(\delta).$$
Since the colors are preserved by the decomposition, the equation $\ell_G(\g)=\ell_G(\tau)+\ell_G(\delta)$, is equivalent to  
\begin{equation}\label{equivalent}
 \inv(\g) + \sum_{c_i(\g)\neq 0} \sigma(i) = \inv(\tau) + \sum_{c_{i}(\tau)\neq 0} | \tau(i)| + \inv(\delta) + \sum_{c_{i}(\delta)\neq 0} |\delta(i)|.\end{equation}
 
For $i>2$ all equations in (\ref{inv-blocs}) hold also in this case. However the last of those fails for the blocs $\g^1$ and $\uparrow\! \g^1$. So we can proceed as in case a) for all entries except for the colored entries of $\g^1$. So let us consider $c_i\neq 0$, with $\s(i)^{c_i} \in \g^1$, and $\sigma(i)=\tau(j)$ with $j\in [d_1]$. Since $\uparrow\! \g^1$ is an order and color preserving reduction of $\g^1$, we have $|\delta(i)|=j$. Hence any colored entry $\s(i)^{c_i} $ contributes $\s(i)$ on the LHS of (\ref{equivalent}), but only $j$ on the RHS. It suffices to show that this difference $\s(i)-j$ is balanced by the extra inversions created by $\tau(j)=\s(i)$ in $\tau$, namely by
\begin{equation*}
\inv(\tau(j) \mid \uparrow\! \g^{2} \cdots \uparrow\! \g^{k}) -\inv(\g(i) \mid \g^{2} \cdots \g^{k}).
\end{equation*}
Note that $\inv(\tau(j) \mid \uparrow\! \g^{2} \cdots \uparrow\! \g^{k})=\inv(\s(i) \mid \g^{2} \cdots \g^{k}$).
From this and the definition of the order \eqref{order}, the above difference is equal to the number of elements in $\g^{2} \ldots \g^{k}$ which  are smaller in absolute value than $\sigma(i)$. This number is equal to $\s(i)-j$ since $\uparrow\! \g^1$ is an increasing sequence and $\s(i)$ is the entry in position $j$ of $\tau$.
\end{itemize}
The proof is now completed.
\end{proof}

\begin{definition}
Let $f=(f_1^{c_1},\ldots,f_n^{c_n}) \in \CS$ and  $\pi(f)$ the associated  colored permutation (cf. Definition~\ref{pif}). We define
\begin{equation}\label{def:lecol}
\inv(f):=\ell_G(\pi(f)) \quad {\rm and} \quad \col(f):=\col(\pi(f)).
\end{equation}
Let  $\underline{n}=(n_0,n_1,\ldots,n_k)$ be a {\em composition} of $n$,  i.e., 
$n_{i}\in \NN$ and $n_{0}+\cdots +n_{k}=n$. We let
\begin{eqnarray*}
\CS(\underline{n}):=\{{f \in \CS \mid \#\{i : f_i=j\}=n_j}\}.
\end{eqnarray*}
\end{definition}

\begin{lemma}\label{analogoReiner}
If $\underline{n}=(n_0,n_1,\ldots,n_k)$ is  a composition of $n$, then
\begin{align}\label{keylem}
\sum_{f \in \CS(\underline{n})} 
p^{\inv(f)}a^{\col(f)}={\hat{n} \brack  \hat{n}_0, n_1, \ldots, n_k}_{a[r-1]_{ap},p}.
\end{align}
\end{lemma}
\begin{proof}
We recall  that for any $f \in \CS$  the descents of $\pi(f)$ can occur only between two different blocs (see \eqref{Anu}). This implies that the restriction of the map $f \mapsto \pi(f)$ induces a bijection between $\CS(\underline{n})$  and the quotient associated to  $J:=[0,n-1]\setminus \{{n_0},{n_0+n_1},\ldots, {n_0+\ldots+n_{k-1}}\}$, namely
\begin{align}\label{bij}
\CS(\underline{n})  \longleftrightarrow G^J=\left\{\g \in G(r,n) \mid \Des_G(\g) \subseteq 
\{n_0,n_0+n_1,\ldots,n_0+\ldots+n_{k-1}\} \right\}.
\end{align}
From (\ref{bij}) and (\ref{def:lecol}) it follows that
$$\sum_{f \in \CS(\underline{n})} p^{\inv(f)}a^{\col(f)} = \sum_{\tau \in G^J} p^{\ell_G(\tau)}a^{\col(\tau)}.$$
 Since  the parabolic subgroup  $G_J$ is isomorphic to $G(r,n_0)\times S_{n_1}\times\ldots \times S_{n_k}$, from Proposition~\ref{quotient} we derive
$$\sum_{\g \in G(r,n)} p^{\ell_G(\g)} a^{\col(\g)} = \sum_{\tau \in G^J}  
p^{\ell_G(\tau)} a^{\col(\tau)}\cdot \sum_{\sigma \in G_J}  p^{\ell_G(\s)} a^{\col(\s)}.$$
Hence
\begin{equation}\label{frazione}
\sum_{\tau \in G^J}  p^{\ell_G(\tau)} a^{\col(\tau)}=\frac{\sum_{\g \in G(r,n)}  p^{\ell_G(\g)} a^{\col(\g)}}{\sum_{\sigma \in G(r,n_0)}  p^{\ell_G(\s)} a^{\col(\s)} \cdot \sum_{\s \in S_{n_1}} p^{\ell_G(\s)} \cdots \sum_{\s \in S_{n_k}} p^{\ell_G(\s)}}.
\end{equation}
To compute $\sum_{\g \in G(r,n)} p^{\ell_G(g)} a^{\col(g)}$ we use again the decomposition of Proposition~\ref{quotient}. This time we take the maximal parabolic subgroup  generated by $J_0:=[1,n-1]$.  It is easy to see that 
$G^{J_0}=\{\tau \in G(r,n) \mid \tau(1)<\ldots<\tau(n)\}$ and $G_{J_0}=S_n$. Hence
$$\sum_{\g \in G(r,n)} p^{\ell_G(\g)} a^{\col(\g)}=\sum_{\tau \in G^{J_0}} p^{\ell_G(\tau)} a^{\col(\tau)}\cdot \sum_{\s \in S_n} p^{\ell_G(\s)}.$$
It  is well known (see (\ref{gf:lengthG})) that the second factor is
$\sum_{\s \in S_n} p^{\ell_G(\s)}=[n]_p!$.
Denote the first factor by  $A_n(p,a):=\sum_{\tau \in G^{J_0}} p^{\ell_G(\tau)} a^{\col(\tau)}$.
Since the position of $n$ in $\tau\in G^{J_0}$ can only be the first if its color is bigger than 0 or the last if it is 0, 
we see that 
$$A_n(p,a)=(1+ap^n[r-1]_{ap})\cdot A_{n-1}(p,a),$$
with $A_0(p,a):=1$. 
By induction we get 
$$A_n(p,a)=\prod_{i=1}^n(1+ap^i[r-1]_{ap}),$$
therefore
\begin{equation}\label{e:ellcol}
\sum_{\g \in G(r,n)} p^{\ell_G(\g)} a^{\col(\g)}=[n]_p! \cdot \prod_{i=1}^n(1+ap^i[r-1]_{ap}).
\end{equation}
By substituting the above results into (\ref{frazione}) we obtain \eqref{keylem}.
\end{proof}
\noindent Note that \eqref{e:ellcol} can be also derived from part 3) of Proposition~\ref{l:projection} together with \cite[Proposition 3.3]{fb94}.

\section{The distribution of $(\des_G,\maj,\ell_G,\col)$}
Reiner~\cite[Theorem 7.2]{re93} computed the generating function of a triple of statistics $(d_{R}, \maj_{R}, \inv_{R})$ of ``type'' number of descents, major index, and length. We refer the reader to his paper for the precise definitions. Reiner's result and Theorem~\ref{teq:wreath} impliy that the equidistributions of the pairs $(d_{R}, \inv_{R})$ and $(\des_{G}, \ell_{G})$ are the same. However, the tri-variate distributions of $(d_{R}, \maj_{R}, \inv_{R})$ and $(\des_{G}, \maj, \ell_{G})$ are different.  
\smallskip

In this section, by using our encoding, the tools developed in \S~\ref{bijection1}, and Lemma~\ref{analogoReiner}, we compute the generating function  of $(\des_G,\maj,\ell_G,\col)$ 
and obtain the following identity. 
\begin{theorem}\label{teq:wreath}  We have 
\begin{align}\label{eq:wreath}
 \sum_{n\geq 0}\frac{u^n}{(t;q)_{n+1}[\hat n]_{a[r-1]_{ap},p}!}\sum_{\g\in G(r,n)}
t^{\des_G(\g)}q^{\maj(\g)}p^{\ell_G(\g)}a^{\col(\g)}=
\sum_{k\geq 0}t^k\prod_{i=0}^{k-1}e[q^{i}u]_{p}\cdot {\hat e}[q^ku]_{a[r-1]_{ap},p}.
\end{align}
\end{theorem}
\begin{proof}
The proof consists in computing in two different ways the series 
\begin{equation*}
\sum_{f\in \CS} t^{\max(f)}q^{\max(f)\cdot n-|f|}
p^{\inv(f)}a^{\col(f)}.
\end{equation*}

First, by using the formula
\begin{equation}\label{tmax}
(1-t)\sum_{k\geq 0}a_{\leq k}t^k=\sum_{k\geq 0}(a_{\leq k}-a_{\leq k-1})t^k=
\sum_{k\geq 0} a_{k}t^k,
\end{equation}
where $a_{\leq k}=a_0+\cdots +a_k$,
we derive immediately
$$
\sum_{f\in \CS}t^{\max(f)}q^{\max(f)\cdot n-|f|}
p^{\inv(f)}a^{\col(f)}=(1-tq^n)\sum_{k\geq 0}t^k
\sum_{f\in \CS \mid \max(f)\leq k} q^{k n-|f|}
p^{\inv(f)}a^{\col(f)}.
$$
By Lemma~\ref{analogoReiner}  we can write the second sum as
\begin{align*}
\sum_{f\in \CS \mid \max(f)\leq k} q^{k\cdot n-|f|}
p^{\inv(f)}a^{\col(f)}
&=\sum_{n_0+\cdots+n_k=n}q^{k \sum_i n_i -\sum_i i n_i}
\sum_{f\in \CS(\underline{n})} 
p^{\inv(f)}a^{\col(f)}\\
&=\sum_{n_0+\cdots+n_k=n}{\hat n \brack \hat n_0, n_1, \ldots, n_k}_{a[r-1]_{ap},p} q^{\sum_i n_i(k-i)},
\end{align*}
which is equal to the coefficient of $u^n$ in 
$$ e[u]_pe[qu]_{p}\ldots {\hat e}[q^ku]_{a[r-1]_{ap},p} \times [\hat n]_{a[r-1]_{ap},p}! .$$
Thus we obtain
\begin{align*}
\sum_{n\geq 0}\frac{u^n}{[\hat n]_{a[r-1]_{ap},p}!(1-tq^n)} & \sum_{f\in \CS} t^{\max(f)}q^{\max(f)\cdot n-|f|} p^{\inv(f)}a^{\col(f)}\\
&=\sum_{k\geq 0}t^k e[u]_p e[qu]_{p}\cdots e[q^{k-1}u]_{p}   {\hat e}[q^ku]_{a[r-1]_{ap},p}.
\end{align*}
\smallskip

On the other hand, by using Lemma~\ref{desmaj} and (\ref{def:lecol}), we obtain
\begin{align*}
\sum_{f\in \CS} t^{\max(f)}q^{\max(f)\cdot n-|f|}
p^{\inv(f)}a^{\col(f)}  \nonumber
&=\sum_{\g \in G(r,n)}p^{\ell_G(\g)}a^{\col(\g)}
\sum_{f \in \CS \mid \pi(f)=\g} t^{\max(f)}q^{\max(f)\cdot n-|f|}\\
&=\sum_{\g \in G(r,n)}\frac{t^{\des_G(\g)}q^{\maj(\g)}p^{\ell_G(\g)}a^{\col(\g)}}{(t;q)_n}.
\end{align*}
By comparing the above two formulas the result follows.
\end{proof}

\begin{remark}\label{r:equivalence}
Note that in the case $r=2$, we obtain the corresponding identity for $B_{n}$.
Conversely, starting from Equation \eqref{eq:wreath}  with $r=2$, thanks to Proposition~\ref{l:projection} we can recover Equation \eqref{eq:wreath} for general $r$, by the substitution $a \leftarrow a[r-1]_{ap}$. Therefore we can say that the two identities are equivalent. In \cite{bz09} we give an independent proof of Theorem~\ref{teq:wreath} for $r=2$. 
\end{remark}

\section{Encoding colored biwords}
In this paragraph we are going to show a result which will be fundamental in the proof of Theorem~\ref{5stat}. It generalizes a well-known bijection of Garsia and Gessel concerning bipartite partitions \cite[Theorem 2.1]{gg79}. For other extensions of this bijection see \cite{berbia06, bl06, bc04, re93}. 

\begin{definition}\label{01} We define ${\mathcal B}(r, n)$ to be the set of colored 
biwords ${g \choose f}={g_1, \ldots, g_n \choose f_1^{c_1},\ldots, f_n^{c_n}} \in \mathcal{P}_n\times \CS$ satisfying the following condition:  
\begin{equation}\label{condition}
 {\rm if} \ g_i=g_{i+1} \ {\rm and} \ \left\{\begin{array}{ll}
 f_i\neq f_{i+1}, & \ {\rm then} \ f_i^{c_i} < f_{i+1}^{c_{i+1}}; \\
 f_i= f_{i+1}, &\ {\rm then} \ c_i=0 \Rightarrow c_{i+1}=0,
\end{array}\right.
\end{equation}

\noindent where $i\geq 0$, and  by convention $g_0:=0$ and $f_0:=0$.
\end{definition}

\begin{remark}\label{r01}
In other words, if $g_i=g_{i+1}$ and $f_i=f_{i+1}$, with $c_i\neq 0$, there are no restrictions on the color of $f_{i+1}$. Note that, in the case $g_i=0$ then condition~(\ref{condition}) implies $c_i(f)=0$.
\end{remark}

\begin{example}\label{ex:biw} The followings are three elements of $\mathcal{B}(3,4)$
$$\begin{pmatrix} 1 & 1 & 3 & 3 \\
4^2 & 4^1 & 6^2 & 0\end{pmatrix}, \quad \begin{pmatrix} 1 & 1 & 3 & 3 \\
4^1 & 4^2 & 6^2 & 0\end{pmatrix}, \quad \begin{pmatrix} 1 & 1 & 3 & 3 \\
4^1 & 4 & 6^2 & 0\end{pmatrix},$$
while
$$\begin{pmatrix} 1 & 1 & 3 & 3 \\
4 & 4^1 & 6^2 & 0\end{pmatrix} \not \in \mathcal{B}(3,4).$$
\end{example}

\begin{theorem}\label{ggg}
There exists a bijection 
$$
{g\choose f} \xrightarrow{\:\varphi\:}  (\g, \lambda, \mu),
$$ 
between the set of colored biwords ${\mathcal B}(r,n)$  and the triplets $  (\g, \lambda, \mu),$
where 
\begin{itemize}
\item[$i)$] $\g\in G(r,n)$, 
\item[$ii)$] $\lambda$ is a partition $\tilde{\g}^{-1}$-compatible.
\item[$iii)$] $\mu$ is a partition $\g$-compatible.
\end{itemize}
\end{theorem}

\begin{proof}
For $g=(g_1,\ldots,g_n) \in \mathcal{P}_n$ and $f=(f_1^{c_1},\ldots, f_n^{c_n})\in \CS$, we define
$\varphi({g\choose f} )=(\g, \lambda, \mu)$ by 
\begin{align}\label{def:varphi}
\g:=\pi(f):=(c_1,\ldots, c_n; \s),\quad \lambda:= g \quad\textrm{and}\quad 
\mu:=(f_{\s(1)},\ldots, f_{\s(n)}).
\end{align}
First of all, we have to show that the triplet $(\g, \lambda, \mu)$ satisfies the above three conditions. To simplify the notation we let $\tilde{c}_i:=c_i(\tilde{\g}^{-1})$, within this proof.

The condition $i)$ is clear. To check that $\la$ is $\tilde{\g}^{-1}$-compatible, we need to show that $\la_i<\la_{i+1}$ for $i \in \Des_G(\tilde{\g}^{-1})$. If $i=0 \in \Des_G(\tilde{\g}^{-1})$, then $\tilde{c}_1>0$, and so $\g=[\ldots,1^{\tilde{c}_1},\ldots]$. Since $\pi(f)=\g$ this implies $c_{1}(f)>0$. 
Since  ${g \choose f} \in \mathcal{B}(r,n)$, Remark~\ref{r01} implies $\la_1=g_{1}>0$. Now, let $i>1$ be a descent of $\tilde{\g}^{-1}$. We have three cases to consider. 
\begin{itemize}
\item[a)] $\s^{-1}(i)>\s^{-1}(i+1)$ and $\tilde{c}_i=\tilde{c}_{i+1}=0$. In this case the window notation of $\g$ will be of the form 
$$\g=[\ldots,(i+1),\ldots, i,\ldots].$$
Since $\pi(f)=\g$ we have $f_{i+1}<f_{i}$. Now ${g \choose f} \in {\mathcal B}(r,n)$ and so this implies $g_i<g_{i+1}$.

\item[b)] $\s^{-1}(i)>\s^{-1}(i+1)$, $\tilde{c}_i=0$, and  $\tilde{c}_{i+1}>0$. In this case 
$$\g=[\ldots,(i+1)^{\tilde{c}_{i+1}},\ldots, i,\ldots].$$
Hence $f_{i+1}\leq f_{i}$. In both cases, equal or strictly smaller, we have $f_{i+1}^{\tilde{c}_{i+1}}<f_{i}$. Since ${g \choose f} \in {\mathcal B}(r,n)$ this implies $g_i<g_{i+1}$.

\item[c)] $\s^{-1}(i)<\s^{-1}(i+1)$,  $\tilde{c}_i\geq 0$, and  $\tilde{c}_{i+1}>0$. In this case 
$$\g=[\ldots,i^{\tilde{c}_i} ,\ldots, (i+1)^{\tilde{c}_{i+1}},\ldots].$$
Hence $f_{i}<f_{i+1}$. Moreover $f_{i+1}^{\tilde{c}_{i+1}}<f_{i}^{\tilde{c}_{i}}$. Since ${g \choose f} \in {\mathcal B}(r,n)$ this implies once again $g_i<g_{i+1}$.
\end{itemize}
Thus part $ii$)  is checked.  Condition  $iii$) holds, since
 $\mu$ is clearly $\g$-compatible in view of (\ref{pro-ass})).  


\smallskip

By construction  the map $\varphi$ is clearly injective. Since any triplet $(\g, \la,\mu)$ satisfying the above three conditions is
equal to  $\varphi({g\choose f})$ with $g:=\la$ and  $f:=\mu^{\tilde{\g}^{-1}}$ (see Proposition~\ref{compatible-associated}), to show that the map  is surjective, it suffices 
 to prove that ${g \choose f} \in B(r,n)$. Hence 
we need to check the condition~\eqref{condition}. So suppose $g_i=g_{i+1}$. Since $g$ is $\tilde{\g}^{-1}$-compatible, by definition
$\tilde{\g}^{-1}(i)<\tilde{\g}^{-1}(i+1),$ so either
\begin{itemize}
\item[1)] $\s^{-1}(i)<\s^{-1}(i+1)$, with $\tilde{c}_{i} \geq 0$ and $\tilde{c}_{i+1}=0$, or 
\item[2)] $\s^{-1}(i)>\s^{-1}(i+1)$, with $\tilde{c}_{i}>0$ and $\tilde{c}_{i+1} \geq0$.
\end{itemize}
In the first case, since $\mu$ is a nondecreasing sequence, we have $f_i=\mu_{\s^{-1}(i)}\leq \mu_{\s^{-1}(i+1)}=f_{i+1}$.  As 
$\tilde{c}_{i+1}=0$ we have $f_i^{\tilde{c}_i}\leq f_{i+1}$, and condition~\eqref{condition} follows.

\noindent In the second case we obtain $f_i \geq f_{i+1}$. If  $f_i > f_{i+1}$ then $f_i^{\tilde{c}_i}< f_{i+1}^{\tilde{c}_{i+1}}$. Otherwise $f_i=f_{i+1}$ and condition~\eqref{condition} holds since $\tilde{c}_{i}>0$.
\end{proof}

\begin{example}
For $r=4$ and $n=7$ we have
$$ 
{g \choose f}=  \begin{pmatrix}
      0&1&1&3&3&4&5  \\
     4&4^1&1&3^3&6&3^1&4^2\\
   \end{pmatrix}
\xrightarrow{\ \varphi \ } 
\left\{\begin{array}{ll}
\g=[3,6^1,4^3,7^2,2^1, 1,5],\\
\lambda=(0,1,1,3,3,4,5),\\
\mu=(1,3,3,4,4,4,6).
\end{array}\right.
$$
We find $\tilde{\g}^{-1}=[6, 5^1, 1, 3^3, 7, 2^1,4^2]$,
$\lambda^\g=(1,4^1, 3^3, 5^2, 1^1, 0,3)$,
and $\mu^{\tilde{\g}^{-1}}=(4,4^1,1,3^3,6, 3^1,4^2)=f$.
Note that $\pi(\lambda^\g)=\tilde{\g}^{-1}$, and $\pi(\mu^{\tilde{\g}^{-1}})=\g$.
\end{example}

\begin{remark} \label{multiset}
It is clear that to each biwords ${g \choose f} \in {\mathcal{B}}(r,n)$, we can associate the multiset of $\left\{{g_i\choose f_i^{c_i}}\right\}_{i\in [n]}$ of its columns.  This multiset  can be decomposed into $r$-disjoint multisets of total cardinality $n$, depending on the colors of $f_i$:
\begin{equation*}
\left\{{g_i\choose f_i} \mid \; g_i\geq 0,\; f_i\geq 0\right\}\cup \;\bigcup_{h=1}^{r-1}
\left\{{g_i\choose f_i^h} \mid  \; g_i>0,\; f_i> 0 \right\}.
\end{equation*}

Conversely given such a multiset  $\left\{{i \choose j^h}^{n_{ij}(h)}\right\}_{0\leq h \leq r-1}$, where $n_{ij}(h)$ is the multiplicity of the column ${i \choose j^h}$, there exist 
${n_{ij} \choose  n_{ij}(1),  \ldots,  n_{ij}(r-1)}$ biwords in $\mathcal{B}(r,n)$ corresponding to it, where $n_{ij}:=\sum_{h=1}^{r-1}n_{ij}(h)$ is the number of colored entries. 
\end{remark}

\section{The distribution of $(\des_G, \ides_G, \maj, \imaj, \col, \icol)$ }
\label{fivevariate}

In this section we compute the distribution  of the vector statistic
$(\des_G, \ides_G, \maj, \imaj, \col, \icol)$ 
over the set of colored permutations $G(r,n)$, and obtain the following identity.

\begin{theorem}\label{5stat} We have
\begin{align}\label{e:6stat}
\sum_{n\geq 0}\frac{u^n}{(t_1;q_1)_{n+1}(t_2;q_2)_{n+1}}
&\sum_{\g \in G(r,n)}
t_1^{\des_G(\g)}t_2^{\des_G(\g^{-1})}q_1^{\maj(\g)}
q_2^{\maj(\g^{-1})} a^{\col(\g)} b^{\col(\g^{-1})} \\
&=\sum_{k_1,k_2\geq 0}\frac{t_1^{k_1}t_2^{k_2}}
{(u ; q_1,q_2)_{k_1+1,k_2+1} (ab  [r-1]_{a, b} u ; q_1,q_2)_{k_1,k_2}}.
\end{align}
where $[r-1]_{a,b}:=\frac{a^{r-1}-b^{r-1}}{a-b}.$
\end{theorem}

\begin{proof}
We will count in two ways the expression
\begin{align}\label{eq:expression}
\sum_{{g\choose f}\in {\mathcal B}(r,n)}
t_1^{\max(f)} t_2^{\max(g)} q_1^{n \max(f)-|f|} q_2^{n \max(g)-|g|} a^{\col(f)} b^{\col(f^{-1})},
\end{align}
where we let $\col(f^{-1}):=\col(\pi(f)^{-1}).$
\smallskip

As before, by using (\ref{tmax}), first rewrite the above sum as
$$
(1-t_1q_1^n)(1-t_2q_2^n)\sum_{k_1,k_2\geq 0}t_1^{k_1}t_2^{k_2}
\sum_{\genfrac{}{}{0pt}{}{{g\choose f}\in {\mathcal B}(r,n)}{\max(f)\leq k_1, \max(g)\leq k_2}}
q_1^{nk_1-|f|}q_2^{nk_2-|g|}a^{\col(f)}b^{\col(f^{-1})}.
$$
By writing  the exponents of $q_1$ (resp. $q_2$) as 
$\sum_s(k_1-|f_s|)$ (resp. $\sum_t(k_2-|g_t|)$) 
 we derive from  Remark~\ref{multiset} that the sum
\begin{align*}
&\sum_{\genfrac{}{}{0pt}{}{{g\choose f}\in {\mathcal B}(r,n)}{\max(f)\leq k_1, \max(g)\leq k_2}} q_1^{nk_1-|f|}q_2^{nk_2-|g|}a^{\col(f)} b^{\col(f^{-1})}
\end{align*}
is equal to the coefficient of $u^n$ in the expansion of 
\begin{align}
&\prod_{0\leq i\leq k_1\atop 0\leq j\leq k_2}\sum_{n_{ij}(0)\geq 0} (uq_1^i q_2^j)^{n_{ij}(0)}  \nonumber \\
& \quad \quad \quad  \quad \times \prod_{0\leq i\leq k_1-1\atop 0\leq j\leq k_2-1} \sum_{n_{ij}\geq 0} \sum_{n_{ij}(1)+\ldots+ n_{ij}(r-1)=n_{ij}} {n_{ij} \choose  n_{ij}(1),  \ldots,  n_{ij}(r-1)}  \prod_{h=1}^{r-1}(ua^h b^{r-h} q_1^iq_2^j)^{n_{ij}(h)}\nonumber \\
&=\prod_{0\leq i\leq k_1\atop 0\leq j\leq k_2}\frac{1}{1-uq_1^{i}q_2^{j}}
\prod_{0\leq i\leq k_1-1\atop 0\leq j\leq k_2-1}\sum_{n_{ij}\geq 0}
\left(u b^r(a/b+\cdots +(a/b)^{r-1})q_1^iq_2^j\right)^{n_{ij}} \label{long} \\
&=\prod_{0\leq i\leq k_1\atop 0\leq j\leq k_2}\frac{1}{1-uq_1^{i}q_2^{j}}
\prod_{0\leq i\leq k_1-1\atop 0\leq j\leq k_2-1} \frac{1}{1- ab [r-1]_{a,b} uq_1^{i}q_2^{j}}\nonumber \\
&=
\frac{1}{(u; q_1,q_2)_{k_1+1,k_2+1}}\frac{1}{(ab[r-1]_{a,b} u; q_1,q_2)_{k_1,k_2}}.\nonumber
\end{align}

Finally,  the sum \eqref{eq:expression} is equal to the coefficient of $u^n$ in
\begin{equation}\label{directcount}
(1-t_1q_1^n)(1-t_2q_2^n)\sum_{k_1,k_2\geq 0}\frac{t_1^{k_1}t_2^{k_2}}
{(u; q_1,q_2)_{k_1+1,k_2+1} (ab[r-1]_{a,b} u; q_1,q_2)_{k_1,k_2}}.
\end{equation}
 
On the other hand, from Theorem~\ref{ggg}, and Proposition~\ref{compatible-associated} the sum 
\eqref{eq:expression} is equal to
\begin{align*}
&\sum_{\g \in G(r,n)}
\sum_{ \mu\in \mathcal{P}_n\atop \pi(\mu^{\tilde{\g}^{-1}})=\g}
t_1^{\max(\mu)} q_1^{n\max(\mu)-|\mu|} a^{\col(\g)} \sum_{ \lambda\in \mathcal{P}_n\atop \pi(\lambda^\g)=\tilde{\g}^{-1}} 
t_2^{\max(\lambda)}q_2^{n\max(\lambda)-|\lambda|} b^{\col(\g^{-1})}\\
&= \sum_{\g \in G(r,n)} \sum_{f \in \NN^{(r,n)} \mid \pi(f)=\g}
t_1^{\max(f)}q_1^{n\max(f)-|f|} a^{\col(\g)}   \sum_{g \in  \NN^{(r,n)} \mid \pi(g)=\tilde{\g}^{-1}}
t_2^{\max(g)}q_2^{n\max(g)-|g|} b^{\col(\g^{-1})} \\
&=\sum_{\g \in G(r,n)}\frac{t_1^{\des_G(\g)} q_1^{\maj(\g)} a^{\col(\g)}}
{(t_1;q_1)_n}
\frac{t_2^{\des_G(\g^{-1})}q_2^{\maj(\g^{-1})}b^{\col(\g^{-1})}}{(t_2;q_2)_n},
\end{align*}
where the last equality follows from 
Proposition~\ref{desmaj}, \eqref{phig}, and Corollary~\ref{l:desmaj}. 
Finally, comparing this expression with (\ref{directcount}) we get the theorem.
\end{proof}

\begin{remark}
As Equation~\eqref{eq:wreath}, Equation~\eqref{e:6stat} is also equivalent to the case of $r=2$. Since $\col(\g^{-1})=nr - \col(\g)$, Equation~\eqref{e:6stat} is clearly equivalent to the $b=1$ case. Now, we can recover the general $r$-case from the $r=2$ case, via the substitution $a \leftarrow a[r-1]_{a}$ and Lemma~\ref{l:projection}.
\end{remark}
\section{Specializations}\label{spec}
In this section we consider four specializations of Identity~(\ref{eq:wreath}) giving 
all possible distributions of pair of statistics. In particular we get generalizations of four classical results dues to Carlitz, Brenti, Reiner and Gessel-Roselle for the symmetric and hyperoctahedral group. Finally show a relation between  Theorem~\ref{5stat}  and a result  of Adin-Roichman.
\bigskip

Letting $p=1$  and substituting $u$ by $([r]_a)u$ in \eqref{eq:wreath}, then 
extracting the coefficient of $u^n/n!$ 
yields a $G(r,n)$ analogue of a result of Chow-Gessel~\cite[Equation (26)]{cg07}:
\begin{align}\label{eq:cg2}
\frac{\sum_{\g \in G(r,n)}t^{\des_G(\g)}q^{\maj(\g)}a^{\col(\g)}}{(t;q)_{n+1}}
=\sum_{k\geq 0}([k+1]_q+a[r-1]_a[k]_q)^nt^k.
\end{align}

 Letting $q\leftarrow q^r$, $p=1$, $a=q$ in \eqref{eq:cg2} yields a formula
for the distribution of descents and flag-major index over $G(r,n)$.
\begin{proposition}[Carlitz's identity for $G(r,n)$]
Let $r, n \in \mathbb{P}$. Then
\begin{align}\label{eq:carlitzG}
\frac{\sum_{\g\in G(r,n)} t^{\des_G(\g)}q^{\fmaj(\g)}}{(t;q^r)_{n+1}}=
\sum_{k\geq 0}t^k [r  k +1]_q^n.
\end{align}
\end{proposition}
The above formula   gives a nice generalization of two  identities  of Carlitz over the symmetric group, and of Chow-Gessel over the hyperoctahedral group.  Moreover it gives a $q$-analogue of \cite[Theorem 17]{stein94} on the Eulerian polynomials of type
$G(r,n)$, due to Steingr\'{\i}msson.

The distribution of descents and length has been computed by Reiner~\cite{re95} for finite and affine Coxeter groups. We extend his results to wreath products as follows.
\begin{proposition}[Reiner identity for $G(r,n)$]
Let $r, n \in \mathbb{P}$. Then
\begin{align}\label{eq:reinerG}
\sum_{n\geq 0} \sum_{\g \in G(r,n)} t^{\des_G(\g)}p^{\ell_G(\g)}
\frac{u^n}{[\hat{n}]_{[r-1]_{p},p}!} = \frac{(1-t) \hat{e}[(1-t)u]_{[r-1]_p,p}}{1-t e[(1-t)u]_p}.
\end{align}
\end{proposition}
It easily follows from Equation~(\ref{eq:wreath}) for $a=q=1$ and by letting $u \leftarrow (1-t)u$.

\smallskip

The distribution of descent and number of negative entries has been computed by Brenti \cite{fb94}. For $p=q=1$, $u \leftarrow u [r]_a (1-t)$ in Equation~\eqref{eq:wreath}, we obtain the following generalization.
\begin{proposition}[Brenti identity for $G(r,n)$]
\begin{align}\label{eq:brentiG}
 \sum_{n\geq 0} \sum_{\g\in G(r,n)}
t^{\des_G(\g)}a^{\col(\g)} \frac{u^n}{n!}=\frac{(1-t) e(u(1-t))}{1- t e((1+a[r-1]_a)u(1-t))}.
\end{align}
\end{proposition}

To compute the generating function of major index and length we proceed as follows.
Setting $a=1$ in equation (\ref{eq:wreath}) yields
\begin{equation}\label{maj-inv}
\sum_{n\geq 0}\frac{ u^n}{(t;q)_{n+1}[\hat{n}]_{[r-1]_p,p}!}\sum_{\g \in G(r,n)}
t^{\des_G(\g)}q^{\maj(\g)}p^{\ell_G(\g)}=
\sum_{k\geq 0}t^k\prod_{i=0}^{k-1}e[q^iu]_{p}\cdot {\hat e}[q^ku]_{[r-1]_{p},p}.
\end{equation}
By multiplying both sides of (\ref{maj-inv}) by $(1-t)$, and then by sending $t\rightarrow 1$ we obtain
\begin{eqnarray*}
\sum_{n\geq 0} \sum_{\g \in G(r,n)}
q^{\maj(\g)}p^{\ell_G(\g)} \frac{u^n}{(q;q)_{n}[\hat{n}]_{[r-1]_p,p}!} & = &
\displaystyle{\prod_{i\geq 0}e[q^iu]_{p}}.
\end{eqnarray*}
Replacing  $u$ by $u/(1-p)$ and applying $q$-binomial formula 
$e[u/(1-p)]_p= \prod_{j \geq 0} \frac{1}{1-p^ju}$
we get the following  $G(r,n)$-analogue  of an identity of Gessel-Roselle (see 
\cite[Theorem 8.5]{ge-thesis}  and the historical note  after Theorem 4.3 in  \cite{agr05}):
\begin{proposition}[Gessel-Roselle identity for $G(r,n)$]

\begin{align}\label{eq:roselle}
\sum_{n\geq 0} \frac{\sum_{\g\in G(r,n)}q^{\maj( \g)}p^{\ell_G(\g)} }{(q;q)_n(-p[r-1]_p;p)_n(p;p)_n} u^n=\frac{1}{(u; p, q)_{\infty, \infty}},
\end{align}
where $\displaystyle{(u; p, q)_{\infty, \infty:}=\prod_{i,j\geq 0} (1-up^{i}q^j)}$.
\end{proposition}

By substitution $q_1 \leftarrow q_1^r$, $q_2 \leftarrow q_2^r$, $a \leftarrow q_1$, and $b\leftarrow q_2$, by multiplying both sizes of \eqref{e:6stat} by $(1-t_1)(1-t_2)$ and by letting $t_1 \rightarrow 1$ and $t_2 \rightarrow 1$ we obtain

\begin{proposition}[Adin-Roichman identity for $G(r,n)$] We have
\begin{align*}
\sum_{n\geq 0}\frac{u^n}{(q_1^r;q_1^r)_{n}(q_2^r;q_2^r)_{n}}
&\sum_{\g \in G(r,n)}
q_1^{\fmaj(\g)}
q_2^{\fmaj(\g^{-1})} = \frac{1}{(u; q_1^r, q_2^r)_{\infty,\infty} (q_1 q_2 [r-1]_{q_1,q_2}u; q_1^r, q_2^r)_{\infty,\infty}}.
\end{align*}
\end{proposition}
Note that this identity gives an explicit formula for the generating function of the Hilbert series of the diagonal invariant algebra DIA, studied by Adin and Roichman in \cite[Theorem 4.1]{ar01}.  
This identity and the previous one (\ref{eq:roselle}), coincide in the case of the symmetric group, that is for $r=1$, while in the general case give rise to two different natural extensions of Gessel-Roselle identity for the symmetric group.


\end{document}